\documentclass[11pt]{article}

\usepackage[english]{babel}
\usepackage{amssymb}
\usepackage{graphics}

\usepackage{amsfonts}

\newcommand{\field}[1]{\mathbb{#1}}

\newcommand{\R}{\field{R}}
\newcommand{\N}{\field{N}}

\newcommand{\Z}{\field{Z}}

\newenvironment{proof}{\par\smallbreak{\sl Proof.~}}
{\unskip\nobreak\hfill \qed \par\medbreak}

\newcommand{\be}{\begin{equation}}

\newcommand{\vp}{\varphi}
\newcommand{\al}{\alpha}
\newcommand{\la}{\lambda}
\newcommand{\om}{\omega}

\newcommand{\A}{{\cal A}}

\newcommand{\Cb}{{\cal C}_{\rm bc}}
\newcommand{\CC}{{\cal C}}
\newcommand{\D}{{\cal D}}
\newcommand{\F}{{\cal F}}
\newcommand{\G}{{\cal G}}
\newcommand{\J}{{\cal J}}
\newcommand{\KK}{{\cal K}}

\newcommand{\de}{\partial}

\newcommand{\ee}{\end{equation}}
\newcommand{\qed}{\hfill \rule{2.3mm}{2.3mm}}
\newtheorem{theorem}{Theorem}[section]
\newtheorem{lemma}[theorem]{Lemma}
\newtheorem{corollary}[theorem]{Corollary}

\newtheorem{remark}[theorem]{Remark}

\newcommand{\reff}[1]{(\ref{#1})}

\title{
Regularity of Time-Periodic Solutions to Autonomous Semilinear Hyperbolic PDEs
} 
\newcounter{thesame}
\author{
	Irina Kmit
	\thanks{Institute of Mathematics, Humboldt University of Berlin, Unter den Linden 6, D-10099 Berlin. On leave from the
		Institute for Applied Problems of Mechanics and Mathematics,
		Ukrainian National Academy of Sciences. 
		{\small   E-mail:
			{\tt kmit@mathematik.hu-berlin.de}
	}}
	\ \ \ Lutz Recke \thanks{Institute of Mathematics, Humboldt University of Berlin, Unter den Linden 6, D-10099 Berlin.
		{\small   E-mail:
			{\tt recke@mathematik.hu-berlin.de}}
}}

\date{}

\begin{document}

\maketitle

\noindent
\begin{abstract}
This paper concerns autonomous boundary value problems for 1D  semilinear hyperbolic PDEs. For time-periodic classical solutions, which satisfy a certain non-resonance condition, we show the following:
If the PDEs are continuous with respect to the space variable $x$ and $C^\infty$-smooth with respect to the unknown function $u$, then the solution is $C^\infty$-smooth with respect to the time variable $t$, and if the PDEs are $C^\infty$-smooth with respect to $x$ and $u$, then the solution is $C^\infty$-smooth with respect to $t$ and $x$. The same is true for appropriate weak solutions.

Moreover, we show examples of time-periodic functions, which do not satisfy the non-resonance condition, such that they are weak, but not classical solutions, and such that they are classical solutions, but not $C^\infty$-smooth, neither with respect to $t$ nor with respect to $x$, even if the PDEs are $C^\infty$-smooth with respect to $x$ and $u$.

For the proofs we use  Fredholm solvability properties of linear time-periodic  hyperbolic PDEs and a result of E. N. Dancer about regularity of solutions to abstract equivariant equations.

\end{abstract}  


\section{Introduction}
\label{Introduction}

In this paper we consider time-periodic solutions to boundary value problems for 1D semilinear first-order hyperbolic systems of the type
$$
\partial_tu_j(t,x)+a_j(x)\partial_xu_j(t,x)=f_j(x,u(t,x))
$$
and 1D semilinear second-order hyperbolic equations
of the type
$$
\partial^2_tu(t,x)-a(x)^2\partial^2_xu(t,x)=f(x,u(t,x),\partial_tu(t,x),\partial_xu(t,x)).
$$

Let us formulate our results concerning first-order systems. Specifically,    we consider $2\times 2$ systems with reflection boundary conditions and time-periodic solutions with period one, i.e. solutions $u=(u_1,u_2)$ to problems of the type
(for $t \in \R$, $x \in [0,1]$)
\be
\label{sys}
\begin{array}{l}
\partial_tu_j(t,x)+a_j(x)\partial_xu_j(t,x)=f_j(x,u(t,x)), \;j=1,2,\\
u_1(t,0)=r_1u_2(t,0),\;
u_2(t,1)=r_2u_1(t,1),\\
u(t+1,x)=u(t,x).
\end{array}
\ee
We suppose that for $j=1,2$
\be
\label{sysass}
\begin{array}{l}
a_j \in C([0,1]),\; r_j \in\R,\;
a_j(x)\not=0 \mbox{ and } a_1(x)\not=a_2(x) \mbox{ for all } x \in [0,1],\\
\partial_{u_1}^k\partial_{u_2}^lf_j 
\mbox{ exist and belong to } C([0,1]\times \R^2)
\mbox{ for all } k,l \in \N \cup \{0\}. \end{array}
\ee
Further, we write (for $t \in \R$, $x,y \in [0,1]$,  and $j=1,2$) 
$$
\al_j(x,y):=\int_x^y\frac{dz}{a_j(z)}.
$$

\begin{theorem}
\label{sysTheorem}
Suppose that \reff{sysass} is fulfilled, and let $u\in  C(\R \times [0,1];\R^2)$ satisfy one of the conditions
\begin{eqnarray}
\label{nonressys}
&&\int_0^1\left(\frac{\partial_{u_1}f_1(x,u(t-\al_1(x,1),x))}{a_1(x)}-
\frac{\partial_{u_2}f_2(x,u(t-\al_2(x,1),x))}{a_2(x)}
\right)dx\nonumber\\
&&\not=\ln |r_1r_2|
\mbox{ for all } t \in \R
\end{eqnarray}
and
\begin{eqnarray}
\label{nonressys1}
&&\int_0^1\left(\frac{\partial_{u_1}f_1(x,u(t+\al_1(0,x),x))}{a_1(x)}-
\frac{\partial_{u_2}f_2(x,u(t+\al_2(0,x),x))}{a_2(x)}
\right)dx\nonumber\\
&&\not=\ln |r_1r_2|
\mbox{ for all } t \in \R.
\end{eqnarray}
Then the following is true:

(i) If $u$ satisfies the boundary and the periodicity conditions in \reff{sys} and if 
there exists a sequence $u^1,u^2,\ldots \in  C^1(\R \times [0,1];\R^2)$ such that, for $j=1,2$,
$$
|u^n_j(t,x)-u_j(t,x)|+
|\partial_tu^n_j(t,x)+a_j(x)\partial_xu^n_j(t,x)-f_j(x,u(t,x))| \to 0 \mbox{ for } n \to \infty
$$
uniformly with respect to $(t,x) \in \R \times [0,1]$, then $u$ is a classical solution to \reff{sys}, in particular, $u$ is $C^1$-smooth.
Moreover, 
all partial derivatives $\partial_t^ku$, $k \in \N$, exist and belong to 
$C(\R \times [0,1];\R^2)$.

(ii) If $u$ is a classical solution to \reff{sys} and if the functions $a_j$ and $f_j$, $j=1,2$, are $C^\infty$-smooth, then $u$ is  $C^\infty$-smooth also.
\end{theorem}

Now we formulate our results concerning time-periodic solutions to second-order equations subjected to one Dirichlet and one Neumann boundary conditions. More precisely, we consider  problems of the type
(for $t \in \R$ and $x \in [0,1]$)
\be
\label{eq}
\begin{array}{l}
\partial^2_tu(t,x)-a(x)^2\partial^2_xu(t,x)=f(x,u(t,x),\partial_tu(t,x),\partial_xu(t,x)),\\
u(t,0)=0,\;
\partial_xu(t,1)=0,\\
u(t+1,x)=u(t,x).
\end{array}
\ee
We assume that
\be
\label{eqass}
\begin{array}{l}
a \in C^1([0,1]),\; 
a(x)\not=0 \mbox{ for all } x \in [0,1],\\
\partial_2^j\partial_3^k\partial_4^lf
 \mbox{ exist and belong to } C([0,1]\times \R^3)
\mbox{ for all } j,k,l \in \N \cup \{0\}, 
\end{array}
\ee
where 
$\partial_jf$ denotes the derivative of the function  $f$ with respect to its $j$-th argument. More precisely, if $f=f(x,u,v,w)$, then
$\partial_2f$ is the derivative  with respect to $u$, 
$\partial_3f$ is the derivative  with respect to $v$, and 
$\partial_4f$ is the derivative  with respect to $w$. 
Further, we write (for $t \in \R$, $x,y \in [0,1]$, and $u \in C([0,1]\times \R^2)$) 
\begin{eqnarray*}
&&\al(x,y):=\displaystyle\int_x^y\frac{dz}{a(z)},\\
&&b_+(t,x,u):= \displaystyle \partial_3f(x,u(t,x),\partial_tu(t,x),\partial_xu(t,x))+\frac{\partial_4f(x,u(t,x),\partial_tu(t,x),\partial_xu(t,x))}{a(x)},\\
&&b_-(t,x,u):=\displaystyle\partial_3f(x,u(t,x),\partial_tu(t,x),\partial_xu(t,x))-\frac{\partial_4f(x,u(t,x),\partial_tu(t,x),\partial_xu(t,x))}{a(x)}.
\end{eqnarray*}
 
The weak formulation of the second-order problem
\reff{eq} (see Theorem \ref{sysTheorem} $(i)$), which will be used, is slightly more   
complicated than that for the first-order problem \reff{sys}, and, in fact, it is a technical tool only. We, therefore, will not include in Theorem \ref{eqTheorem} below a regularity result for weak solutions to \reff{eq},  but only regularity results for classical solutions to \reff{eq}.
\begin{theorem}
\label{eqTheorem}
Suppose that \reff{eqass} is fulfilled. Let $u\in  C^2(\R \times [0,1])$ be a classical solution to \reff{eq}, and suppose that it satisfies one of the conditions
\be
\label{nonreseq}
\int_0^1\frac{b_{+}
(t+\al(x,1),x,u)
-b_{-}
(t-\al(x,1),x,u)}{a(x)}\,dx\not=0
\mbox{ for all } t \in \R
\ee
and
\be
\label{nonreseq1}
\int_0^1\frac{b_{+}
(t-\al(0,x),x,u)
-b_{-}
(t+\al(0,x),x,u)}{a(x)}\,dx\not=0
\mbox{ for all } t \in \R.
\ee
Then the following is true:

(i) All partial derivatives $\partial_t^ku$, $k \in \N$, exist and belong to 
$C(\R \times [0,1])$.

(ii) If the functions $a$ and $f$ are $C^\infty$-smooth, then $u$ is  $C^\infty$-smooth also.
\end{theorem}

\begin{remark}
\label{Hopf}\rm
In most applications, solutions to problems of the type \reff{sys} and \reff{eq} are found as a result of Hopf bifurcations from stationary solutions  \cite{KR3,KR3a,Kos1} and by continuation of such solutions with respect to parameters \cite{KR4}.
\end{remark}

\begin{remark}
\label{Hale}\rm
The paper \cite{HS} of J. K. Hale and J. Scheurle
concerns smoothness with respect to time of solutions to abstract autonomous semilinear evolution equations if those solutions are bounded and close to be constant in time. The results are applied to slightly damped nonlinear wave equations in 1D with constant coefficients, namely
\be
\label{HS}
\partial^2_tu(t,x)-\partial^2_xu(t,x)+\delta \partial_tu(t,x)-u(t,x)-\lambda u(t,x)=f(u(t,x)),
\ee
subjected to homogeneous Dirichlet boundary conditions.
The function $f:\R \to \R$ is smooth and of order $o(|u|)$ for $u \to 0$, $\lambda$ is small, $\delta$ is positive and small.
It is shown that sufficiently small bounded solutions are smooth with respect to time.

Let us compare this with Theorem \ref{eqTheorem}:
On one hand, the equation in our problem \reff{eq} is more general than equation \reff{HS}.
Moreover,  in Theorem \ref{eqTheorem} we do not suppose that the solution is close to be constant in time. On the other hand, our Theorem \ref{eqTheorem} concerns time-periodic solutions only, not general bounded ones. 
Anyway, if one applies definitions of the functions $b_+$ and $b_-$ to equation \reff{HS}, then
$$
b_+(t,x,u)=b_-(t,x,u)=-\delta.
$$
Hence, the assumption $\delta>0$ of \cite{HS} implies that the assumptions of Theorem \ref{eqTheorem} are fulfilled.
\end{remark}

\begin{remark}
\label{telegraph}\rm
Let us consider Theorem \ref{eqTheorem} in the special case of a nonlinear wave equation which is slightly more general than \reff{HS}, namely
$$
\partial^2_tu(t,x)-a(x)^2\partial^2_xu(t,x)=\beta_1(x)\partial_tu(t,x)+\beta_2(x) \partial_xu(t,x)+f(x,u(t,x)).
$$
If one applies definitions of $b_+$ and $b_-$ to this equation, then
$$
b_+(t,x,u)=\beta_1(x)+\frac{\beta_2(x)}{a(x)},\;
b_-(t,x,u)=\beta_1(x)-\frac{\beta_2(x)}{a(x)}.
$$
Hence, the conditions \reff{nonreseq}  and \reff{nonreseq1} are identical, and they are satisfied for any $u$ if and only if
$$
\int_0^1\frac{\beta_1(x)}{a(x)}\,dx\not=0.
$$
\end{remark}

\begin{remark}
\label{multidim}\rm
We do not know if Theorems \ref{sysTheorem} and \ref{eqTheorem} can be generalized to cases of  
more than one space dimension.
The reason is that linear autonomous hyperbolic partial differential operators with one space dimension essentially differ from those with more than one space dimension:
They satisfy the spectral mapping property in $L^p$-spaces \cite{Lopes} and, which is more important
for applications to nonlinear problems,  in $C$-spaces \cite{Lichtner1}.
Moreover,  they generate Riesz bases (see, e.g. \cite{Guo,Mogul}). This is  not the case,
in general, if the space dimension is larger than one (see the counter-example of M. Renardy in \cite{Renardy}).
Therefore, the question of  Fredholmness of those operators in appropriate spaces of time-periodic functions is highly difficult.
\end{remark}


\begin{remark}
\label{nn}\rm
Theorem \ref{sysTheorem} can be generalized to problems for $n \times n$ first-order hyperbolic systems of the type
(with natural numbers $m<n$)
\be 
\label{mn}
\begin{array}{l}
\partial_tu_j(t,x)+a_j(x)\partial_xu_j(t,x)=f_j(x,u(t,x)), \;j\le n,\\
\displaystyle
u_j(t,0)=\sum_{k=m+1}^nr_{jk}u_k(t,0),\; j\le m,\\
\displaystyle
u_j(t,1)=\sum_{k=1}^mr_{jk}u_k(t,1),\, m<j\le n.
\end{array}
\ee
Here, instead of non-resonant  conditions \reff{nonressys} and
\reff{nonressys1}, one considers  the following sufficient conditions
$$
\max_{s,t \in [0,1]}
\max_{ j \le m}\sum_{k=m+1}^n\sum_{l=1}^m|r_{jk}r_{kl}|\exp\int_0^1
\left(\frac{\partial_{u_k}f_k(x,u(t,x))}{a_k(x)}-
\frac{\partial_{u_j}f_j(x,u(s,x))}{a_j(x)}\right)\,dx<1
$$
and
$$
\max_{s,t \in [0,1]}
\max_{m< j \le n}\sum_{k=1}^m\sum_{l=m+1}^n|r_{jk}r_{kl}|\exp\int_0^1
\left(\frac{\partial_{u_k}f_k(x,u(t,x))}{a_k(x)}-
\frac{\partial_{u_j}f_j(x,u(s,x))}{a_j(x)}
\right)\,dx<1.
$$
If one of these two conditions is satisfied for a function $u$, then the linearization in $u$ of problem \reff{mn} has Fredholm solvability properties (cf. \cite{KR4}).

	In \cite{smoothing}, using different approach,  the issue of higher  regularity of time-periodic solutions to general {\it linear nonautonomous} first-order hyperbolic systems, namely to
systems
\begin{equation}\label{eq:1}
	(\partial_t  + a(x,t)\partial_x + b(x,t))
	u = f(x,t), 
\end{equation}
subjected to {\it nonlinear} reflection boundary conditions of the type
\be\label{eq:dis}
\begin{array}{ll}
	u_j(0,t) = h_j(z(t)), & 1\le j\le m, \ \ 
	\\
	u_j(1,t) = h_j(z(t)), & m< j\le n, \ \  
\end{array}
\ee
where
\be\label{z}
z(t)=\left(u_1(1,t),\dots,u_{m}(1,t),u_{m+1}(0,t),\dots,u_{n}(0,t)\right)\nonumber
\ee 
is addressed. It is shown that  continuous solutions are $C^l$-regular
whenever $l$ conditions of the type
\be\label{contr2}
\exp \left\{\int_x^{x_j}
\left(\frac{b_{jj}}{a_{j}}-r\frac{\partial_ta_j}{a_{j}^2}\right)(\eta,\om_j(\eta;x,t))\,d\eta\right\}\sum_{k=1}^n
\left|\partial_kh_j^\prime(z)\right|<1
\ee
for all $j\le n$, $x\in[0,1]$, $t\in\R$, $z\in\R^n$, and $r=0,1,\dots,l$, are fulfilled.  It turns out that 
 the nonautonomous setting essentially relates the 
solution regularity with the number of conditions of the type \reff{contr2} (see
\cite[Remark 1.4]{KR2} and \cite[Subsection 3.6]{KRT3}).
\end{remark}

\begin{remark}
\label{nonauto}\rm
In \cite{KR4} we consider the question of smoothness with respect to time of time-periodic solutions to non-autonomous semilinear problems of the type
$$
\partial_tu_j(t,x)+a_j(x)\partial_xu_j(t,x)=f_j(t,x,u(t,x)), \;j=1,2,
$$
and 
$$
\partial^2_tu(t,x)-a(x)^2\partial^2_xu(t,x)=f(t,x,u(t,x),\partial_tu(t,x),\partial_xu(t,x)). 
$$
In \cite{KR4} it is supposed that the linearized problems (linearization in the solution to the nonlinear problem) do not have nontrivial solutions. This is essentially different to the autonomous case, because in the autonomous case the linearization in the solution $u$ to the nonlinear problem has $\partial_tu$ as solution. On the other hand, in the non-autonomous case this assumption concerning the linearized problem implies not only regularity with respect to time of the solution to the nonlinear problem, but also its local uniqueness and smooth dependence on parameters.

 In  \cite{second} we studied higher regularity of  time-periodic solutions to non-autonomous linear problems for the equation
 $$
 \partial^2_tu-a(t,x)^2\partial^2_xu+a_1(t,x)\de_tu + a_2(t,x)\de_xu + a_3(t,x)u =f(t,x). 
 $$ 
 We showed that any additional order of regularity requires additional non-resonance conditions.
\end{remark}

The remaining part of the paper is organized as follows:

Section \ref{proofssys} concerns
problem \reff{sys} for  first-order systems and Section \ref{proofeq} concerns problem~\reff{eq} for second-order equations.

In Subsection \ref{weak} we introduce an appropriate weak formulation of problem \reff{sys} such that Corollary \ref{app1} is applicable to it. In Subsection \ref{Fredholm} we show that the main assumption of Corollary \ref{app1}, which is Fredholmness of the linearized problem, is fulfilled whenever one of the conditions \reff{nonressys} and \reff{nonressys1}
is satisfied.
Using this, we prove Theorem \ref{sysTheorem} in Subsection~\ref{Proofsys}.

In Subsection \ref{Transfo} we show that the second-order problem \reff{eq} is equivalent to a first-order problem of the type \reff{sys}, but with additional nonlocal integral terms in the equations and in the boundary conditions.

Finally, in Appendix we provide Theorem \ref{app} and Corollary \ref{app1} from abstract  nonlinear analysis.

\section{Proofs for first-order systems}
\label{proofssys}
\setcounter{equation}{0}
\setcounter{theorem}{0}
In this section we will prove  Theorem \ref{sysTheorem}. Hence, we will suppose that  assumption \reff{sysass} is satisfied. 

We will work with the function space
$$
\CC:=\{u \in C(\R\times[0,1];\R^2):\; u(t+1,x)=u(t,x) \mbox{ for all } t \in \R, \; x \in [0,1]\},
$$
which is equipped and complete with the maximum norm 
$$
\|u\|_\infty:=\max \{|u_j(t,x)|:\; t \in \R, \; x \in [0,1],\; j=1,2\},
$$
and with the function space
$\CC^1:=\{u \in \CC:\; \mbox{$u$ is $C^1$-smooth}\}$, which is equipped and complete with the norm $\|u\|_\infty+\|\partial_tu\|_\infty+\|\partial_xu\|_\infty$. Further, we will work with the closed subspaces
$$
\Cb:=\{u \in \CC: \; u_1(t,0)=r_1u_2(t,0), u_2(t,1)=r_2u_1(t,1) \mbox{ for all } t \in \R\},\;
\Cb^1:=\Cb\cap \CC^1
$$
in $\CC$ and $\CC^1$, respectively.
Finally, we consider the linear bounded operator
$$
A:\CC^1\to\CC:\; Au:=(\partial_tu_1+a_1\partial_xu_1,\partial_tu_2+a_2\partial_xu_2)
$$
and the nonlinear $C^\infty$-smooth superposition operator
$$
F:\CC\to \CC:\; 
[F(u)](t,x):=(f_1(x,u(t,x)),f_2(x,u(t,x))).
$$
Obviously, a function $u$ is a classical solution to problem \reff{sys} if and only if it is a solution to the problem
\be
\label{absys}
u \in \Cb^1: \; Au=F(u).
\ee

Now we aim at applying Corollary \ref{app1} to problem \reff{absys}. To this end,  we introduce the one-parameter group $S_\vp \in {\cal L}(\CC)$, $\vp \in \R$, which is defined by
$$
[S_\vp u](t,x):=u(t+\vp,x).
$$
It is easy to verify that $S_\vp$ is strongly continuous on $\CC$ as well as on $\CC^1$, i.e. that the map $\vp \mapsto S_\vp u$ is continuous from $\R$ into $\CC$ for all $u \in \CC$ and that
this map is continuous from $\R$ into $\CC^1$ for all $u \in \CC^1$. Moreover, we have
$$
S_\vp Au=AS_\vp u,\; S_\vp F(u)=F(S_\vp u).
$$
Hence, Corollary \ref{app1} is applicable (with $U=\Cb^1$, $V=\CC$, and $\F(u)=Au-F(u)$)
to all solutions of \reff{absys} such that
$A-F'(u)$ is Fredholm of index zero from $\Cb^1$ into $\CC$. But, unfortunately, $A-F'(u)$ is not Fredholm of index zero from $\Cb^1$ into $\CC$, no matter if $u$ satisfies one of the conditions \reff{nonressys} and \reff{nonressys1} or not.
The reason for that is, roughly speaking, the following: The domain of definition $\Cb^1$ is slightly too small. It should be enlarged properly, or, in other words, we should work with an appropriate weak formulation of \reff{sys}.

\subsection{Weak formulation}
\label{weak}
Let $\bar A: \D(\bar{A})\subseteq \CC \to \CC $ denote
 the closure of the linear operator $A$.
The appropriate weak formulation of \reff{sys} is
\be
\label{absys1}
u \in \D(\bar{A})\cap\Cb: \;\bar Au=F(u).
\ee

\begin{lemma}
\label{closable}
The operator $A$ is closable in $\CC$.
\end{lemma}
\begin{proof}
Take a sequence $u^1,u^2,\ldots \in \CC^1$ and 
$v \in \CC$ such that $\|u^n\|_\infty+\|Au^n-v\|_\infty \to 0$ for $n \to \infty$. 
We have to show that $v=0$.

From  $\partial_y \al_1(x,y)=1/a_1(y)$  it follows that
\begin{eqnarray*}
&&u^n_1(t,x)-u^n_1(t+\al_1(x,0),0)=\int_0^x
\frac{d}{dy}u^n_1(t+\al_1(x,y),y)\,dy\\
&&=\int_0^x(\partial_tu^n_1(t+\al_1(x,y),y)+a_1(y)\partial_xu^n_1(t+\al_1(x,y),y))\frac{dy}{a_1(y)}. 
\end{eqnarray*}
Taking the limit as $n \to \infty$, we get $\int_0^xv_1(t,y)/a_1(y)\,dy=0$. It follows that $v_1(t,x)/a_1(x)=0$,
i.e. $v_1=0$. Here we used the assumption that $a_1(x) \not=0$ for all $x \in [0,1]$ (cf. \reff{sysass}).

Similarly one shows that $v_2=0$.
\end{proof}

The closure $\bar A$ of $A$ is, by definition, the smallest closed extension of $A$ in $\CC$.
The domain of definition $\D(\bar A)$ of $\bar A$ is the set of all $u \in \CC$ such that there exist a sequence $u^1,u^2,\ldots \in \CC^1$ and an element $v \in \CC$ such that 
\be 
\label{closure}
\|u^n-u\|_\infty+\|Au^n-v\|_\infty \to 0
\mbox{ for } n \to \infty, 
\ee
and $\bar A$ works on $u \in \D(\bar A)$ as $\bar A u:=v$. Because of Lemma \ref{closable}, this definition is correct, i.e. independent of the choice of $u^1,u^2,\ldots$ and $v$ with \reff{closure}.

\begin{remark}
\label{larger}\rm
For 1-periodic continuous functions $\phi:\R \to \R$ and $y \in [0,1]$, define $u_{\phi,y} \in \CC$ by
$$
u_{\phi,y}(t,x):=(\phi(t+\al_1(x,y)),\phi(t+\al_2(x,y))).
$$
Then $u_{\phi,y}\in \D(\bar A)$: Indeed, take a sequence of 1-periodic $C^1$-functions $\phi^1,\phi^2,\ldots:\R \to \R$, such that $\phi^n(t)\to \phi(t)$ for $n \to \infty$ uniformly with respect to $t \in \R$. Then $u_{\phi^n,y} \in \CC^1$ and $Au_{\phi^n,y}=0$. Hence, \reff{closure} is satisfied with $u=u_{\phi,y}$,
$u^n=u_{\phi^n,y}$ and $v=0$. But $u_{\phi,y} \notin \CC^1$, if $\phi$ is not $C^1$-smooth. Hence, $\D(\bar A)$ is larger than $\CC^1$, i.e. $\bar A$ is a proper extension of $A$.
\end{remark}

\begin{remark}
\label{counterex}\rm
Consider problem \reff{sys} with $a_1(x)\equiv 4$, $a_2(x)\equiv -4$, $f_1(x,u)\equiv f_2(x,u)\equiv 0$ and $r_1=-r_2=1$, i.e.
\be
\label{sysex}
\begin{array}{l}
\displaystyle\partial_tu_1(t,x)+4\partial_xu_1(t,x)=
\partial_tu_2(t,x)-4\partial_xu_2(t,x)=0,\\
u_1(t,0)=u_2(t,0),\;
u_2(t,1)=-u_1(t,1),\\
u(t+1,x)=u(t,x).
\end{array}
\ee
For continuous functions $\phi:\R \to \R$, which satisfy $\phi(t+1/2)=-\phi(t)$ for all $t \in \R$, define $u_{\phi} \in \CC$ by
$$
u_{\phi}(t,x):=(\phi(t-x/4),\phi(t+x/4)).
$$
Then $u_{\phi}\in \D(\bar A)\cap \Cb$. Hence, if $\phi$ is not $C^1$-smooth, then $u_{\phi}$ is a solution to \reff{absys1}, but not to \reff{absys}. Similarly, if $\phi$ is $C^1$-smooth, but not $C^2$-smooth,
then $u_{\phi}$ is a classical solution to \reff{sysex}, but $u_{\phi}(\cdot,x)$ is not $C^2$-smooth. This is possible because $u_{\phi}$  satisfies neither \reff{nonressys} nor \reff{nonressys1}.
\end{remark}

It is well-known that the domain of definition of a closed linear operator, equipped  with the graph norm, 
is complete.
Hence, the function space $\D(\bar A)\cap \Cb$, equipped  with the norm
$\|u\|_\infty+\|\bar A u\|_\infty$, is a Banach space.
Moreover, the shift operators $S_\vp$ constitute a $C_0$-group of linear bounded operators on this Banach space. Hence, problem \reff{absys1} is  a candidate for an application of Corollary \ref{app1} (with $U=\D(\bar A)\cap \Cb$, $V=\CC$ and $\F(u)=\bar Au-F(u)$). All the more this is true because of the following lemma.

\begin{lemma}
\label{Fred}
Let a function $u \in \CC$ satisfy one of the conditions
\reff{nonressys} and \reff{nonressys1}. Then the operator $\bar A-F'(u)$ is Fredholm of index zero from $\D(\bar A)\cap \Cb$ (equipped with the graph  norm) into $\CC$.
\end{lemma}
 
Now we prepare the proof of Lemma \ref{Fred}, which will be done in Subsection \ref{Fredholm}. As
in the proof of Lemma \ref{closable},  we will use  integration along characteristics.

For any given $u \in \CC$, we introduce linear bounded operators $B(u), C(u), D(u):\CC \to \CC$ by
$$
[B(u)v](t,x):=\left[
\begin{array}{c}
\partial_{u_1}f_1(x,u(t,x))v_1(t,x)\\
\partial_{u_2}f_2(x,u(t,x))v_2(t,x)
\end{array}
\right],
$$
$$
[C(u)v](t,x):=\left[
\begin{array}{c}
\displaystyle r_2 v_2(t+\al_1(x,0),0)
\exp\left(\int_0^x\frac{\partial_{u_1}f_1(z,u(t+\al_1(x,z),z))}{a_1(z)}\,dz\right) \\
\displaystyle r_1v_1(t+\al_2(x,1),1)  \exp\left(-\int_x^1\frac{\partial_{u_2}f_2(z,u(t+\al_2(x,z),z))}{a_2(z)}\,dz\right) 
\end{array}
\right],
$$
and
$$
[D(u)v](t,x):=\left[
\begin{array}{c}
	\displaystyle
\int_0^x\frac{v_1(t+\al_1(x,y),y)}{a_1(y)}
\exp\left(\int_y^x\frac{\partial_{u_1}f_1(z,u(t+\al_1(x,z),z))}{a_1(z)}\,dz\right) 
dy,\\
\displaystyle-\int_x^1\frac{v_2(t+\al_2(x,y),y)}{a_2(y)}
\exp\left(\int_y^x\frac{\partial_{u_2}f_2(z,u(t+\al_2(x,z),z))}{a_2(z)}\,dz\right) 
dy\end{array}
\right]
$$

\begin{lemma}
\label{BCD}
(i) For all $u,v \in \CC$, we have $C(u)v \in \D(\bar A)$ and 
$(\bar A-B(u))C(u)v=0$.

(ii) For all $u,v \in \CC$, we have $D(u)v \in \D(\bar A)$ and 
$(\bar A-B(u))D(u)v=v$.

(iii) For all $u \in \CC$ and $v \in \D(\bar A) \cap \Cb$, we have
$D(u)(\bar A-B(u))v=v-C(u)v$.

(iv)  If for functions $u,v \in \CC$ the partial derivatives $\partial_tu$ and $\partial_tv$ exist and are continuous, then the functions $C(u)v$ and $D(u)v$ are $C^1$-smooth.

(v)  If the functions $a_j$ 
and $f_j$, $j=1,2$, are $C^\infty$-smooth,
then for any $k \in \N$ the following is true: 
If for functions $u,v \in \CC$ all partial derivatives $\partial^l_t\partial_x^mu$ and $\partial^l_t\partial_x^mv$, with $l \in \N$ and
$m\le k$, exist and are continuous, then the partial derivatives $\partial^l_t\partial_x^m
C(u)v$ and $\partial^l_t\partial_x^mD(u)v$, with $l \in \N$ and
$m\le k+1$, exist and are continuous. 
\end{lemma}
\begin{proof} $(i)$ Let $u,v \in \CC$ be given. We have to show that there exists a sequence $w^1,w^2,\ldots \in \CC$
such that
\be
\label{wconv}
\|w^n- C(u)v\|_\infty \to 0 \mbox{ for } n \to \infty
\ee
and
\be
\label{Awconv}
\|Aw^n- B(u)C(u)v\|_\infty \to 0 \mbox{ for } n \to \infty.
\ee
We construct this sequence as follows:
Because of $\CC^1$ is dense in $\CC$, there exist
sequences $u^1,u^2,\ldots \in \CC^1$ and
$v^1,v^2,\ldots \in \CC^1$
such that 
\be 
\label{uvconv}
\|u^n-u\|_\infty+ \|v^n-v\|_\infty\to 0 \mbox{ for } n \to \infty.
\ee
Therefore, we can choose $w^n:=C(u^n)v^n$. Then
\reff{wconv} is satisfied. It remains to prove \reff{Awconv}. For that reason we calculate 
\begin{eqnarray*}
&&[A_1w^n](t,x)\\
&&=(\partial_t+a_1(x)\partial_x)\left( r_2 v^n_2(t+\al_1(x,0),0)
\exp\left(\int_0^x\frac{\partial_{u_1}f_1(z,u^n(t+\al_1(x,z),z))}{a_1(z)}\,dz\right)\right)\\
&&=r_2 v^n_2(t+\al_1(x,0),0)
\exp\left(\int_0^x\frac{\partial_{u_1}f_1(z,u^n(t+\al_1(x,z),z))}{a_1(z)}\,dz\right)\partial_{u_1}f_1(x,u^n(t,x))\\
&&=\left[B_1(u^n)C(u^n)v^n\right](t,x).
\end{eqnarray*}
Similarly one shows that $A_2w^n=B_2(u^n)C(u^n)v^n$, i.e. $Aw^n=B(u^n)C(u^n)v^n$. This implies~\reff{Awconv}.

$(ii)$ Similar to $(i)$, take sequences $u^1,u^2,\ldots \in \CC^1$ and
$v^1,v^2,\ldots \in \CC^1$
with \reff{uvconv}, and set $w^n:=D(u^n)v^n$.
Then
\begin{eqnarray*}
&&[A_1w^n](t,x)\\
&&=(\partial_t+a_1(x)\partial_x)
\displaystyle\int_0^x\frac{v_1^n(t+\al_1(x,y),y)}{a_1(y)}
\exp\left(-\int_x^y\frac{\partial_{u_1}f_1(z,u^n(t+\al_1(x,z),z))}{a_1(z)}\,dz\right) 
dy\\
&&=v^n_1(t,x)+[B_1(u^n)D(u^n)v^n](t,x).
\end{eqnarray*}
Similarly one shows that $A_2w^n=B_2(u^n)D(u^n)v^n$, i.e.
$Aw^n=B(u^n)D(u^n)v^n$. Hence,
$$
\|Aw^n-B(u)D(u)v\|_\infty \to 0 \mbox{ for } n \to \infty.
$$

$(iii)$  Let $u \in \CC$ and $v \in \D(\bar A) \cap \Cb$ be given. Take a sequence $v^1,v^2,\ldots \in \CC^1$ such that $\|v^n-v\|_\infty 
+\|Av^n-\bar A v\|_\infty 
\to 0$ for $n \to \infty$. Then
\begin{eqnarray*}
&&v^n_1(t,x)-[C_1(u)v^n](t,x)\\
&&=(v^n_1(t+\al_1(x,0),0)-r_2v^n_2(t+\al_1(x,0),0))
\exp\left(\int_0^x\frac{\partial_{u_1}f_1(z,u(t+\al_1(x,z),z))}{a_1(z)}\,dz\right)\\
&&=\int_0^x\partial_y\left(
v^n_1(t+\al_1(x,y),y)
\exp\left(\int_y^x\frac{\partial_{u_1}f_1(z,u(t+\al_1(x,z),z))}{a_1(z)}\,dz\right)\right)dy\\
&&=[D_1(A-B(u))v^n](t,x).
\end{eqnarray*}
Taking the limit as $n \to \infty$ and using the boundary condition for $v$ in $x=0$, we get
$v_1-C_1(u)v=D_1(\bar A-B(u))v$. Similarly one shows that $v_2-C_2(u)v=D_2(\bar A-B(u))v$.

$(iv)$ This assertion follows directly from the definitions of the operators $C(u)$ and $D(u)$.

$(v)$ Suppose that the functions $a_j$ and $f_j$ with $j=1,2$ are $C^\infty$-smooth. Take an integer $k \in \N$ and functions $u,v \in \CC$ such that all partial derivatives $\partial^l_t\partial_x^mu$ and $\partial^l_t\partial_x^mv$, for $l \in \N$ and
$m\le k$, exist and are continuous.

In order to show that all 
partial derivatives $\partial^l_t\partial_x^m
C_1(u)v$, for $l \in \N$ and
$m\le k+1$, exist and are continuous (and similarly for all 
partial derivatives $\partial^l_t\partial_x^m C_2(u)v$),
it suffices to show that  all 
partial derivatives $\partial^l_t\partial_x^m$, for $l \in \N$ and
$m\le k+1$, of the functions
\be 
\label{v2}
(t,x) \in \R \times [0,1] \mapsto v_2(t+\al_1(x,0),0) \in \R 
\ee
and 
\be 
\label{f1}
(t,x) \in \R \times [0,1] \mapsto \int_0^x \frac{\partial_{u_1}f_1(z,u(t+\al_1(x,z),z))}{a_1(z)}\,dz
\in \R 
\ee
exist and are continuous. This is obvious
for \reff{v2}, and for \reff{f1} it follows from the fact that all 
partial derivatives $\partial^l_t\partial_x^m$, for $l \in \N$ and
$m\le k$, of the functions
\be 
\label{f1a}
(t,x) \in \R \times [0,1] \mapsto \frac{\partial_{u_1}f_1(x,u(t,x))}{a_1(z)}\,dz
\in \R 
\ee
exist and are continuous.

The claim  that all 
partial derivatives $\partial^l_t\partial_x^m
D_1(u)v$, for $l \in \N$ and
$m\le k+1$, exist and are continuous (and similar for $\partial^l_t\partial_x^m
D_2(u)v$) follows from \reff{f1a} and the the obvious fact that  all 
partial derivatives $\partial^l_t\partial_x^m$, for $l \in \N$ and
$m\le k$, of the functions
$$
(t,x) \in \R \times [0,1] \mapsto \frac{v_1(t,x)}{a_1(x)} \in \R 
$$
exist and are continuous.

\end{proof}

\begin{lemma}
\label{Invert}
Let a function $u \in \CC$ satisfy one of the conditions \reff{nonressys}
and \reff{nonressys1}. Then the operator $I-C(u)$ is bijective from $\CC$ to $\CC$.
\end{lemma}
\begin{proof} 
Let $u,f \in \CC$. Assume that $u$
satisfies one of the conditions \reff{nonressys}
and \reff{nonressys1}.
We have to show that  there exists a unique solution $v \in \CC$ to the equation
  \be
  \label{1a}
  v=C(u)v+f.
  \ee
  For $t \in \R$, $x,y \in [0,1]$, and $j=1,2$, set 
\be
\label{cdef}
  c_j(t,x,y):=\exp\left(\int_y^x\frac{\partial_{u_j}f_j(z,u(t+\al_j(x,z),z))}{a_j(z)}\,dz\right). 
  \ee
Equation \reff{1a} is satisfied if and only if for 
  all $t \in \R$ and $x \in [0,1]$ we have
  \begin{eqnarray}
    \label{2a}
    v_1(t,x)&=&r_1c_1(t,x,0)v_2(t+\al_1(x,0),0)+f_1(t,x),\\
     \label{3a}
    v_2(t,x)&=&r_2c_2(t,x,1)v_1(t+\al_2(x,1),1)+f_2(t,x).
    \end{eqnarray}
 System \reff{2a}--\reff{3a} is satisfied if and only 
 if \reff{2a} is true and if 
  \begin{eqnarray}
     \label{4a}
    v_2(t,x)&=&r_1r_2c_1(t+\al_2(x,1),1,0)
    c_2(t,x,1)v_2(t+\al_1(1,0)+\al_2(x,1),0)
    \nonumber\\
    &&+r_2c_2(t,x,1)f_1(t+\al_2(x,1),1)+
    f_2(t,x).
  \end{eqnarray}
 Put $x=0$ in \reff{4a} and get
   \begin{eqnarray}
     \label{5a}
    v_2(t,0)&=&r_1r_2c_1(t+\al_2(0,1),1,0)
    c_2(t,0,1)u_2(t+\al_1(1,0)+\al_2(0,1),0)
    \nonumber\\
    &&+r_2c_2(t,0,1)f_1(t+\al_2(0,1),1)+f_2(t,0).
   \end{eqnarray}
Similarly,  system \reff{2a}--\reff{3a} 
is satisfied if and only 
 if \reff{3a} is true and 
  \begin{eqnarray}
     \label{4b}
    v_1(t,x)&=&r_1r_2c_2(t+\al_1(x,0),0,1)
    c_1(t,x,0)v_1(t+\al_2(0,1)+\al_1(x,0),0)
    \nonumber\\
    &&+r_2c_2(t+\al_1(x,0),0,1)f_2(t+\al_1(x,0),0)+
    f_1(t,x),
  \end{eqnarray}
  i.e. if and only if  \reff{3a} and \reff{4b} are true and 
   \begin{eqnarray}
     \label{5b}
   v_1(t,1)&=&r_1r_2c_2(t+\al_1(1,0),0,1)
  c_1(t,1,0)v_1(t+\al_2(0,1)+\al_1(1,0),0)
  \nonumber\\
  &&+r_2c_2(t+\al_1(1,0),0,1)f_2(t+\al_1(1,0),0)+
  f_1(t,1).
   \end{eqnarray}
   
Let us consider equation \reff{5a}. It is a functional equation for the unknown function $v_2(\cdot,0)$. In order to solve it, let us denote
   by $C_{per}(\R)$ the Banach space of all $1$-periodic continuous functions $\tilde{v}:\R \to\R$ with the norm
   $\|\tilde{v}\|_\infty:=\max\{|\tilde{v}(t)|: \; t \in \R\}$.  Equation \reff{5a} is an equation in  $C_{per}(\R)$ of the type
   \be
   \label{6a}
   (I-\widetilde{C})\tilde{v}=\tilde{f}
   \ee
   with $\tilde{v},\tilde{f}\in C_{per}(\R)$ defined by $\tilde{v}(t):=v_2(t,0)$ and
   $$
   \tilde{f}(t):=r_2c_2(t,0,1)
   f_1(t+\al_2(0,1),1)+f_2(t,0)
   $$
   and with $\widetilde{C}\in {\cal L}(C_{per}(\R))$ defined by
   $$
   [\widetilde{C}\tilde{v}](t):=
 r_1r_2c_1(t+\al_2(0,1),1,0)
    c_2(t,0,1)\tilde{v}(t+\al_1(1,0)+\al_2(0,1)).
   $$
   From the definitions of the functions $c_1$ 
   and $c_2$ it follows that
   \begin{eqnarray*}
   &&c_1(t+\al_2(0,1),1,0)c_2(t,0,1)\\
   &&=\exp \int_0^1\left( \frac{\partial_{u_1}f_1(t+\al_1(1,x)+\al_2(0,1),x)}{a_1(x)}-\frac{\partial_{u_2}f_2(t+\al_2(0,x),x)}{a_2(x)}
\right)\,dx
\end{eqnarray*}
 and, hence,
\begin{eqnarray*}
   &&c_1(s+\al_2(0,1),1,0)
    c_2(s,0,1)|_{s=t-\al_2(0,1)}\\
   &&=\exp \int_0^1\left(\frac{\partial_{u_1}f_1(t-\al_1(x,1),x)}{a_1(x)}-\frac{\partial_{u_2}f_2(t-\al_2(x,1),x)}{a_2(x)}
 \right)\,dx,  
\end{eqnarray*}
Consequently, if assumption \reff{nonressys} is satisfied, then
$|r_1r_2c_1(t+\al_2(0,1),1,0)
    c_2(t,0,1)|\not=1$ for all $t \in \R$.

First, let us consider the case that 
$$
c_+:=\max\{|r_1r_2c_1(t+\al_2(0,1),1,0)
    c_2(t,0,1)|:\; t \in \R\}<1.
$$
 Then 
   $$
   \|\widetilde{C}\|_{ {\cal L}(C_{per}(\R))}\le \frac{1+c_+}{2}<1.
   $$
   Hence, the operator $I-\widetilde{C}$ is an isomorphism from 
   $C_{per}(\R)$ to itself.
   Therefore, there exists a unique solution $v_2(\cdot,0)\in C_{per}(\R)$.
   Inserting this solution into the right-hand sides of \reff{4a} 
   and \reff{2a}, we get the unique solution $v=(v_1,v_2) \in \CC$ to \reff{2a}--\reff{3a}.
  
    Now, let us consider the case that 
 $$
c_-:=\min\{|r_1r_2c_1(t+\al_2(0,1),1,0)
    c_2(t,0,1)|:\; t \in \R\}>1.
$$   
Then 
   $$|r_1r_2c_1(t+\al_2(0,1),1,0)
    c_2(t,0,1)|
   \ge \frac{1+c_-}{2}\ge 1.$$ 
Equation \reff{5a} is equivalent to
 \begin{eqnarray*}
      v_2(t,0)&=&
    \frac{v_2(t-\al_1(1,0)-\al_2(0,1),0)}
    {r_1r_2c_1(t-\al_1(1,0),1,0)
    c_2(t-\al_1(1,0)-\al_2(0,1),0,1)}\\
    &&-\frac{f_1(t-\al_1(1,0),0)}{r_1c_1(t-\al_1(1,0)-\al_2(0,1),1,0)}\\
    &&-\frac{f_2(t-\al_1(1,0)-\al_2(0,1),0)}{r_1r_2c_1(t-\al_1(1,0),1,0)
    c_2(t-\al_1(1,0)-\al_2(0,1),0,1)}.
 \end{eqnarray*}
 This equation is of the type \reff{6a} again, but now with 
 $$
 \|\widetilde{C}\|_{ {\cal L}(C_{per}(\R))}\le\frac{2}{1+c_-}<1.
 $$
Hence, we can proceed as above.

Similarly one deals with the case if condition \reff{nonressys1} is satisfied.  Then equation \reff{5b} is uniquely solvable, and so is equation \reff{4b} and, hence, system \reff{2a}--\reff{3a}.\end{proof}

\begin{corollary}
\label{iso}
Let $u \in \CC$ satisfy one of the conditions
\reff{nonressys} and \reff{nonressys1}. Then the operator $\bar A-B(u)$ is bijective from $\D(\bar A)\cap \Cb$ (equipped with the operator norm) to $\CC$, and
\be
\label{invformula}
(\bar A-B(u))^{-1}v=(I-C(u))^{-1}D(u)v
\mbox{ for all } v \in \CC.
\ee
\end{corollary}
\begin{proof} To show the injectivity, suppose that $(\bar A-B(u))v=0$ for some $v \in \D(\bar A)\cap \Cb$. Then Lemma \ref{BCD} implies that $(I-C(u))v=0$, and Lemma \ref{Invert} yields $v=0$.

To show the surjectivity and inversion formula \reff{invformula},
take $f \in \CC$. Because of Lemma~\ref{Invert}, there exists  $v \in \CC$ such that
$$
v=C(u)v+D(u)f.
$$
In particular, $v \in \Cb$ (cf. the definitions of the operators $C(u)$ and $D(u)$).
Moreover, Lemma~\ref{BCD} $(i)$ and $(ii)$ yields that $v \in \D(\bar A)$ and
$$
(\bar A-B(u))v= (\bar A-B(u))(C(u)v+D(u)f)
=f.
$$
\end{proof}

\begin{remark}
\label{langua}
\rm
Let us explain where the name ``non-resonance condition'' comes from.

Corollary \ref{iso} claims that,
if $u \in \CC$ satisfies one of the conditions
\reff{nonressys} and \reff{nonressys1}, then for any $g \in \CC$ there exists exactly one solution $v$ to the problem
\be
\label{sysev}
\begin{array}{l}
\partial_tv_j(t,x)+a_j(x)\partial_xv_j(t,x)-\partial_{u_j}f_j(x,u(t,x))v_j(t,x)=g_j(t,x), \;j=1,2,\\
v_1(t,0)=r_1v_2(t,0),\;
v_2(t,1)=r_2v_1(t,1),\\
v(t+1,x)=v(t,x).
\end{array}
\ee
Suppose that the function $u$ is independent of time, i.e. $u(t,x)=u(x)$, and let $b_j(x):=
\partial_{u_j}f_j(x,u(x))$. It is easy to calculate that the eigenvalues to the eigenvalue problem
\begin{eqnarray*}
&&a_j(x)v'_j(x)-b_j(x)v_j(x)=\la v_j(x), \;j=1,2,\\
&&v_1(0)=r_1v_2(0),\;
v_2(1)=r_2v_1(1)
\end{eqnarray*}
are
$$
\la_k=\displaystyle\frac{\ln|r_1r_2|-\displaystyle
\int_0^1\left(
\frac{b_2(x)}{a_2(x)}-\frac{b_1(x)}{a_1(x)}
\right)dx}{\displaystyle\int_0^1\left(\frac{1}{a_2(x)}-\frac{1}{a_1(x)}\right)dx}+2k\pi i,\; k \in \Z.
$$
Hence, all eigenvalues have non-vanishing real parts if and only if
$$
\ln|r_1r_2|\not=
\int_0^1\left(
\frac{b_2(x)}{a_2(x)}-\frac{b_1(x)}{a_1(x)}
\right)dx,
$$
and this is just condition \reff{nonressys} or condition \reff{nonressys1}
(in the case that the coefficients $\partial_{u_j}f_j(x,u(t,x))$ are independent of time, \reff{nonressys} and \reff{nonressys1} are the same).
In this case all "internal frequencies" $\la_k/2\pi i$, $k \in \Z$, of system \reff{sysev} are different to all "external frequencies" $k \in \Z$ of the right-hand sides $g_j$, and one says that the external frequencies are not in resonance with the internal frequencies.
 \end{remark}

\subsection{Proof of Lemma \ref{Fred}}
\label{Fredholm}
Suppose that $u \in \CC$ satisfies one of the conditions \reff{nonressys}
and \reff{nonressys1}. Write
$$
\widetilde B (u):=F'(u)-B(u).
$$
We have to show that the operator $\bar A-F'(u)=
\bar A-B(u)-\widetilde B (u)$ is Fredholm of index zero from $\D(\bar A) \cap \Cb$ (equipped with the graph norm) into $\CC$. Because of Corollary \ref{iso}, this is the case if and only if the operator
$$
(\bar A-B(u))^{-1}(\bar A-F'(u))=
I-(\bar A-B(u))^{-1}\widetilde B (u)=
I-(I-C(u))^{-1}D(u)\widetilde B (u)
$$
is Fredholm of index zero from $\CC$ into $\CC$.
Hence, it suffices to show that 
$$
\left[(I-C(u))^{-1}D(u)\widetilde B (u)\right]^2 \mbox{ is compact from $\CC$ into $\CC$.} 
$$
 This is a consequence of the following Fredholmness criterion of S. M. Nikolskii (cf. e.g.  \cite[Theorem XIII.5.2]{Kant}):
If $U$ is a Banach space and $K:U\to U$ is a linear bounded operator   such that $K^2$ is compact, then
   the operator $I-K$ is Fredholm of index zero.
   
Since $u$ is fixed, in what follows in this subsection we will not mention the dependence of the operators $B(u)$, $\widetilde B (u)$, $C(u)$ and $D(u)$ on $u$, i.e. $B:=B(u)$, $\widetilde B:=\widetilde B (u)$, $C:=C(u)$, and $D:=D(u)$.
A straightforward calculation shows that
\be
\label{ident}
\left[(I-C)^{-1}D\widetilde B\right]^2
=(I-C)^{-1}
\left[(D\widetilde B)^2+D\widetilde B C(I-C)^{-1}D\widetilde B\right].
\ee
Then, on account of Lemma \ref{Invert},  it suffices to show that the operators $D\widetilde B D$ and $D\widetilde B C$ are compact
from $\CC$ into $\CC$.

Let us show that $D_1\widetilde B D$  (and similarly for $D_2\widetilde B D$)
is compact from $\CC$ into $\CC$. 
Take $v \in \CC$. By definition, $B$ and $\widetilde B$ are the "diagonal" and the "non-diagonal" parts of $F'(u)$. Therefore,
$$
[\widetilde B v](t,x)= 
\Bigl(\partial_{u_2}f_1(x,u(t,x))v_2(t,x),\partial_{u_1}f_2(x,u(t,x))v_1(t,x)\Bigr).
$$
Hence, the first component of 
$[D\widetilde B v](t,x)$ is
$$
[D_1\widetilde B v](t,x)=\int_0^x\frac{c_1(t,x,y)}
{a_1(y)}
\partial_{u_2}f_1(y,u(t+\al_1(x,y),y))v_2(t+\al_1(x,y),y)\,dy.
$$
Therefore,
 \be
\label{D1BD}
[D_1\widetilde B Dv](t,x)=\int_0^x\int_y^1d(t,x,y,z) v_2(t+\al_1(x,y)+\al_2(y,z),z)\,dzdy
\ee
with
$$
d(t,x,y,z):=-\frac{c_1(t,x,y)c_2(t+\al_1(x,y),x,z)}{a_1(y)a_2(z)}
\partial_{u_2}f_1(y,u(t+\al_1(x,y),y)).
$$
We change the order of integration in \reff{D1BD} according to
\be
\label{change}
\int_0^xdy\int_y^1dz=\int_0^xdz\int_0^zdy+\int_x^1dz\int_0^xdy.
\ee

Let us consider the first summand in the right-hand side of \reff{change}. It is
the linear operator
\be
\label{firsts}
[{\cal K}v](t,x):=\int_0^x\int_0^zd(t,x,y,z) v(t+\al_1(x,y)+\al_2(y,z),z)\,dydz.
\ee
We have to show that  ${\cal K}$  is compact from $C([0,1]^2)$ (equipped with the maximum norm) into itself.
For that reason we replace
in the inner integral in \reff{firsts}  the integration variable $y$ with  a new integration variable $\eta$ according to
$$
\eta=\widehat{\eta}(t,x,y,z):=t+\al_1(x,y)+\al_2(y,z)=t+\int_x^y\frac{d\xi}{a_1(\xi)}+\int_y^z\frac{d\xi}{a_2(\xi)}.
$$
Because of the assumption that $a_1(x)\not=a_2(x)$ for all $x \in [0,1]$ (cf. \reff{sysass}), we have
$$
\partial_y\widehat{\eta}(t,x,y,z)=\frac{1}{a_1(y)}-\frac{1}{a_2(y)}\not=0
\mbox{ for all } y \in [0,1],
$$
i.e. the function $y \mapsto \widehat{\eta}(t,x,y,z)$ is strictly monotone. Let us denote its inverse function by $\eta \mapsto \widehat{y}(t,x,\eta,z)$. Then 
\be
\label{Idef}
[{\cal K}v_2](t,x)=\int_0^x\int_{\widehat{\eta}(t,x,0,z)}^{\widehat{\eta}(t,x,z,z)}
\tilde{d}(t,x,\eta,z)v_2(\eta,z)\,d\eta dz
\ee
with
$$
\displaystyle\tilde{d}(t,x,\eta,z):=\frac{d(t,x,\widehat{y}(t,x,\eta,z),z)}{\displaystyle\frac{1}{a_1(\widehat{y}(t,x,\eta,z))}-\frac{1}{a_2(\widehat{y}(t,x,\eta,z))}}. 
$$
Due to assumption \reff{sysass},
the function $\tilde{d}$ is continuous, and the function $\hat \eta$ is $C^1$-smooth.
Hence, the Arcela-Ascoli Theorem implies that the linear operator ${\cal K}$ is compact from $C([0,1]^2)$, 
equipped with the maximum norm, into itself.

Similarly one shows that also the second summand in the right-hand side of \reff{change}, which is
$$
\int_x^1\int_0^xd(t,x,y,z) v_2(t+\al_1(x,y)+\al_2(y,z),z)\,dydz,
$$ 
generates a compact operator from $C([0,1]^2)$ into itself.

Finally, let us show that the operator $D\widetilde B C$ 
is compact from $\CC$ into itself.
We have (and similarly for $D_2\widetilde B C$)
$$
[D_1\widetilde B Cv](t,x)=\int_0^xd(t,x,y) v_1(t+\al_1(x,y)+\al_2(y,1),1)\,dy
$$
with
$$
d(t,x,y):=r_2\frac{c_1(t,x,y)c_2(t+\al_1(x,y),x,y)}{a_1(y)}
\partial_{u_1}f_2(y,u(t+\al_1(x,y),y)).
$$
Here we change the integration variable $y$ to $\eta=t+\al_1(x,y)+\al_2(y,1)$, and then proceed as above.

\subsection{Proof of Theorem \ref{sysTheorem}}
\label{Proofsys}
Take a function $u \in C(\R \times [0,1];\R^2)$ which satisfies the boundary and the periodicity conditions as in \reff{sys} and such that 
there exists a sequence $u^1,u^2,\ldots \in  C^1(\R \times [0,1];\R^2)$ satisfying the 
following convergence for $j=1,2$:
$$
|u^n_j(t,x)-u_j(t,x)|+
|\partial_tu^n_j(t,x)+a_j(x)\partial_xu^n_j(t,x)-f_j(x,u(t,x))| \to 0 \mbox{ for } n \to \infty,
$$
uniformly in $x\in[0,1]$ and $t\in\R$.
Then $u$ is a solution to \reff{absys1}, and Lemma \ref{BCD} $(iii)$ yields
\be
\label{ueq}
u=C(u)u+D(u)(F(u)-B(u)u).
\ee
Further, we suppose that $u$ satisfies one of the conditions \reff{nonressys} and \reff{nonressys1}.
Then, due to  Lemma~\ref{Fred} and Corollary \ref{app1},  all partial derivatives
$\partial_t^ku$, $k \in \N$, exist and are continuous. Therefore, all partial derivatives
$\partial_t^k$, $k \in \N$, of the functions $F(u)$ and $B(u)u$ exist and are continuous also. 
Hence, Lemma \ref{BCD} $(iv)$ and \reff{ueq} yield that $u \in \Cb^1$ and $\bar A u=Au$, i.e. $u$ is a classical solution to \reff{sys}. 
Assertion $(i)$  of Theorem \ref{sysTheorem} is therefore proved.

Similarly, if the functions $a_j$ and $f_j$, $j=1,2$, are $C^\infty$-smooth, then Lemma \ref{BCD} $(v)$ and \reff{ueq} yield that $u$ is $C^\infty$-smooth, i.e. assertion $(ii)$ of Theorem \ref{sysTheorem} is proved also.

\section{Proofs for second-order equations}
\label{proofeq}
\setcounter{equation}{0}
\setcounter{theorem}{0}
In this section we will prove  Theorem \ref{eqTheorem}. Hence, we  suppose that  assumption \reff{eqass} is satisfied.

\subsection{Transformation of the second-order equation into a first-order system}
\label{Transfo}
In this subsection we show that any solution $u$ to  problem \reff{eq} for a second-order equation creates a solution 
\be 
\label{trafo}
v_1(t,x):=\de_tu(t,x)+a(x)\de_xu(t,x),\quad
v_2(t,x):=\de_tu(t,x)-a(x)\de_xu(t,x)
\ee
to the following problem for a first-order  system of integro-differential
equations:
\be
\label{FOS}
  \begin{array}{l}
\de_tv_1(t,x)-a(x)\de_xv_1(t,x)=
\de_tv_2(t,x)+a(x)\de_xv_2(t,x)\\
\displaystyle  =
f(x,[J v](t,x), [Kv](t,x),[L v](t,x))+\frac{a'(x)}{2}(v_1(t,x)-v_2(t,x)),\\ [3mm]
v_1(t,0)+v_2(t,0)=
v_1(t,1)-v_2(t,1)=0,\\
    v(t+1,x)=v(t,x),
  \end{array}
\ee
and vice versa. Here the partial integral operator $J$ is defined by
$$
[Jv](t,x):=\frac{1}{2}\int_0^x\frac{v_1(t,y)-v_2(t,y)}{a(y)}dy,
$$
and the "local" operators $K$ and $L$ are defined by 
$$
[Kv](t,x):=\frac{v_1(t,x)+v_2(t,x)}{2}, \quad [Lv](t,x):=\frac{v_1(t,x)-v_2(t,x)}{2a(x)}.
$$

\begin{lemma}\label{lem:FOS}
(i) If $u \in C^2(\R\times [0,1])$ is a solution to \reff{eq}, then the function $v \in C^1(\R\times [0,1];\R^2)$
 defined by~\reff{trafo}
is a solution to \reff{FOS}.

(ii) Let $v \in C^{1}(\R\times [0,1];\R^2)$ be a solution to \reff{FOS}. Then the function  $u:=Jv$
is $C^2$-smooth and is a solution to \reff{eq}.
\end{lemma}
\begin{proof}
  $(i)$ Let $u \in C^2(\R\times [0,1];\R^2)$ be given, and let $v \in 
 C^1_{per}(\R\times [0,1];\R^2)$ be defined by~\reff{trafo}. Then
\be
  \label{KJ}
  \de_tu=\frac{v_1+v_2}{2}=Kv,\quad  \de_xu=\frac{v_1-v_2}{2a}=L v
  \ee
and $\de_tv_1=\de^2_tu+a\de_t\de_xu$,  
$\de_xv_1=\de_t\de_xu+a' \de_xu + a\de_x^2u$, $\de_tv_2=\de^2_tu-a\de_t\de_xu$, and  $\de_xv_2=\de_t\de_xu-a' \de_xu - a\de_x^2u$.
Hence,
\be
  \label{uv}
  \de_t^2u-a^2\de_x^2u-aa'\de_xu=
  \de_tv_1-a\de_xv_1=
  \de_tv_2+a\de_xv_2.
  \ee
Further, let $u$ be a solution to problem \reff{eq}.  
Then \reff{KJ} and the boundary conditions $u(t,0)=0$ and $\de_xu(t,1)+\gamma u(t,1)=0$ imply that $v_1(t,0)+v_2(t,0)=0$ and
  $v_1(t,1)-v_2(t,1)+\gamma a(1) [Lv](t,1)=0$, i.e. the boundary conditions as in \reff{FOS}.
Moreover, from $u(t,0)=0$ and \reff{KJ} follows  that 
$u(t,x)=[Jv](t,x)$.
  Hence, \reff{KJ}, \reff{uv}, and the differential equation in \reff{eq} yield the differential equations as in \reff{FOS}.

 $ (ii)$ Let $v \in C^{1}(\R\times [0,1];\R^2)$ be a solution to \reff{FOS}. Set $u:=Jv$.
Then
\begin{eqnarray*}
  \de_tu(t,x)&=&
  \frac{1}{2}\int_0^x\frac{\de_tv_1(t,y)-\de_tv_2(t,y)}{a(y)}\,dy\\
  &=&
 \int_0^x\left(\de_yv_1(t,y)+\de_yv_2(t,y)\right)dy=
 v_1(t,x)+v_2(t,x).
  \end{eqnarray*}
  Here we used the first boundary condition and the differential equations in \reff{FOS}.
It follows that $\de_tu$ is $C^{1}$-smooth, and
  \be
  \label{uv1}
  \de^2_tu=\de_tv_1+\de_tv_2.
  \ee
  Further, the relation $u=Jv$ yields that $\de_xu=\de_x Jv=L v$,
  i.e.  $\de_xu$ is $C^{1}$-smooth also and, hence  $u$ is $C^{2}$-smooth. Moreover,
   $2(a'\de_xu+a\de_x^2u)=\de_xv_1-\de_xv_2$, i.e.
   \be
  \label{uv3}
  a^2\de^2_xu=\frac{a}{2}(\de_xv_1-\de_xv_2)-\frac{a'}{2}(v_1-v_2).
  \ee
  But \reff{FOS}, \reff{uv1}, and \reff{uv3} imply
the differential equation as in \reff{eq}. 

The first boundary condition in \reff{eq} follows from $u=Jv$, and 
the second boundary conditions in \reff{eq}
follows from $\de_xu=Lv$ and from the second boundary condition in \reff{FOS}.

\end{proof}

Unfortunately, we cannot apply  Theorem \ref{sysTheorem}  directly to system \reff{FOS} because there are nonlocal terms in the equations in \reff{FOS}. Hence, we adapt the content of Section \ref{proofssys} to the situation of system \reff{FOS}.

\subsection{Weak formulation of \reff{FOS}}
\label{weakeq}
We use the notation of $\al(x,y)$, $b_+(t,x,u)$ and
$b_-(t,x,u)$, which were introduced in Section~ \ref{Introduction}, as well as the function spaces $\CC$ and $\CC^1$, which were introduced in Section \ref{proofssys}.
Further, we introduce a linear bounded operator
$A:\CC^1\to\CC$ by
$$
Av:=(\partial_tv_1-\partial_xv_1,\partial_tv_2+\partial_xv_2)
$$
and a nonlinear $C^\infty$-smooth superposition operator $F:\CC\to \CC$ by 
$$
[F_j(v)](t,x):=f\left(x,[Jv](t,x),[Kv](t,x),[Lv](t,x)\right)
+\frac{a'(x)}{2}(v_1(t,x)-v_2(t,x)),\quad j=1,2.
$$
Any classical solution to \reff{FOS} is a solution to the problem
$Av=F(v)$
and, hence, a solution 
to the problem
\be
\label{abeq}
v \in \D(\bar A) \cap \Cb:\; \bar Av=F(v).
\ee

In order to apply Corollary \ref{app1} to problem 
\reff{abeq} (with $U=\D(\bar A) \cap \Cb$, $V=\CC$, and $\F(v)=\bar Av-F(v)$),
we proceed as in Section \ref{proofssys}.
We write (for $t \in \R$, $x \in [0,1]$, $v \in \CC$)
\begin{eqnarray*}
c_+(t,x,v)&:=& b_+(t,x,Jv)\\
&=&\partial_3f(x,[Jv](t,x),[Kv](t,x),[Lv](t,x))
+\frac{\partial_4f(x,[Jv](t,x),[Kv](t,x),[Lv](t,x))}{a(x)},\\
c_-(t,x,v)&:=& b_-(t,x,Jv)\\
&=&\partial_3f(x,[Jv](t,x),[Kv](t,x),[Lv](t,x))
-\frac{\partial_4f(x,[Jv](t,x),[Kv](t,x),[Lv](t,x))}{a(x)}.
\end{eqnarray*} 
Note that, if a function $u\in C^1(\R \times [0,1])$ satisfies condition \reff{nonreseq}, then the function $v\in C(\R \times [0,1];\R^2)$, which is defined by \reff{trafo}, satisfies the condition
\be
\label{nonreseqa}
\int_0^1\frac{c_{+}
(t+\al(x,1),x,v)
-c_{-}
(t-\al(x,1),x,v)}{a(x)}\,dx\not=0\quad
\mbox{ for all } t \in \R.
\ee
Similarly, if $u$ satisfies \reff{nonreseq1}, then $v$ satisfies the condition
\be
\label{nonreseqa1}
\int_0^1\frac{c_{+}
(t-\al(0,x),x,u)
-c_{-}
(t+\al(0,x),x,u)}{a(x)}\,dx\not=0\quad
\mbox{ for all } t \in \R.
\ee

We divide the linearization $F'(v)$ into three parts, a ``diagonal'' one, a ``non-diagonal'' one, and an ``integral'' one, as follows:
\be 
\label{divide} 
F'(v)=B(v)+\widetilde B (v)+\J(v)
\ee
with
$$
[B(v)w](t,x):=\frac{1}{2}\left[
\begin{array}{c}
(c_+(t,x,v)+a'(x))
w_1(t,x)\\
(c_-(t,x,v)-a'(x))
w_2(t,x)
\end{array}
\right],
$$
$$
[\widetilde B(v)w](t,x):=\frac{1}{2}\left[
\begin{array}{c}
(c_-(t,x,v)-a'(x))
w_2(t,x)\\
(c_+(t,x,v)+a'(x))
w_1(t,x)
\end{array}
\right],
$$
and 
$$
[\J_1(v)w](t,x)=[\J_2(v)w](t,x):=
\partial_2f(x,[Jv](t,x),[Kv](t,x),[Lv](t,x))
[Jw](t,x).
$$
As in Subsection \ref{weak}, for given $v \in \CC$,  we introduce linear bounded operators $C(v),D(v):\CC \to \CC$  by
$$
[C(v)w](t,x):=\left[
\begin{array}{c}
\displaystyle  -w_2(t-\al(x,0),0)
\sqrt{\frac{a(0)}{a(x)}}
\exp\left(-\frac{1}{2}\int_0^x\frac{c_+(t-\al(x,z),z,v)}{a(z)}\,dz\right) \\
\displaystyle  w_1(t+\al(x,1),1)
\sqrt{\frac{a(1)}{a(x)}}
\exp\left(-\frac{1}{2}\int_x^1\frac{c_-(t+\al(x,z),z,v)}{a(z)}\,dz\right) 
\end{array}
\right]
$$
and
$$
[D(v)w](t,x):=\left[
\begin{array}{c}
\displaystyle
-\frac{1}{\sqrt{a(x)}}
\int_0^x\frac{w_1(t-\al(x,y),y)}{\sqrt{a(y)}}
\exp\left(\frac{1}{2}\int_x^y\frac{c_+(t-\al(x,z),z,v)}{a(z)}\,dz\right)dy,\\
\displaystyle
-\frac{1}{\sqrt{a(x)}}
\int_x^1\frac{w_2(t+\al(x,y),y)}{\sqrt{a(y)}}
\exp\left(\frac{1}{2}\int_y^x\frac{c_-(t+\al(x,z),z,v)}{a(z)}\,dz\right)dy
\end{array}
\right].
$$
Here we adapted the definitions of $C(u)v$ and $D(u)v$ from Subsection \ref{weak} as follows: We replaced $u$ by $v$, $v$ by $w$, $a_1$ by $-a$, $a_2$ by $a$, $r_1$ by minus one, $r_2$ by one, $\partial_{u_1}f_1(x,u(t,x))$ by $\frac{1}{2}(c_+(t,x,v)+a'(x))$ and $\partial_{u_2}f_2(x,u(t,x))$ by $\frac{1}{2}(c_-(t,x,v)-a'(x))$. We used also the identity
$$
\exp\left(\frac{1}{2}\int_x^y\frac{a'(z)}{a(z)}\,dz\right)=\sqrt{\frac{a(y)}{a(x)}}.
$$
Remark that, if in \reff{nonressys} and in \reff{nonressys1} we replace
$a_1$ by $-a$, $a_2$ by $a$, $r_1$ by minus one, $r_2$ by one, 
$\partial_{u_1}f_1(x,u(t,x))$ by $\frac{1}{2}(c_+(t,x,v)+a'(x))$, and $\partial_{u_2}f_2(x,u(t,x))$ by $\frac{1}{2}(c_-(t,x,v)-a'(x))$, we immediately get \reff{nonreseqa} and \reff{nonreseqa1}.

Finally, the subspace of all functions, which satisfy the boundary conditions as in \reff{FOS}, is
$
\Cb:=\{v \in \CC:\;
v_1(t,0)+v_2(t,0)=
v_1(t,1)-v_2(t,1)=0\}.
$

Similar to  Lemmas \ref{BCD} and  \ref{Invert} and Corollary \ref{iso}, we get the following:
\begin{lemma}
\label{isoeq}
If $v \in \CC$ satisfies one of the conditions \reff{nonreseqa}
and \reff{nonreseqa1}, then $I-C(v)$ is bijective from $\CC$ onto $\CC$, 
 $\bar A-B(v)$ is bijective from $\D(\bar A)\cap \Cb$ onto $\CC$, and
$$
(\bar A-B(v))^{-1}w=(I-C(v))^{-1}D(v)w\quad
\mbox{for all } w \in \CC.
$$
\end{lemma}

\subsection{Fredholmness}
\label{Fredholmneq}
\begin{lemma}
\label{Fredeq1}
Let a function $v \in \CC$ satisfy one of the conditions
\reff{nonreseqa} and \reff{nonreseqa1}. Then the operator $I-C(v)-D(v)(\widetilde B(v)+\J(v))$ is Fredholm of index zero from $\CC$ to itself.
\end{lemma}
\begin{proof}
We proceed as in the proof of Lemma \ref{Fred}.
We have to show that the operator 
$I-(I-C(v))^{-1}\left[D(v)(\widetilde B (v)+\J(v))\right]$
is Fredholm of index zero from $\CC$ into $\CC$.
The compactness criterion of Nikolskii implies that it suffices to show that 
\be 
\label{Fredeq}
\Big((I-C(v))^{-1}\left[D(v)(\widetilde B (v)+\J(v))\right]\Big)^2 \mbox{ is compact from $\CC$ into $\CC$.} 
\ee
As $v$ is fixed,  we will drop the dependence of the operators $B(v)$, $\widetilde B (v)$, 
$C(v)$, $D(u)$, $E(v)$, and $\J(v)$ on $v$, i.e. $B:=B(v)$, $\widetilde B:=\widetilde B (v)$, $C:=C(v)$, $D:=D(v)$ and $\J=\J(v)$.
As in Subsection \ref{Fredholm},  we use the formula
$$
\left((I-C)^{-1}(D(\widetilde B+\J))\right)^2
=(I-C)^{-1}\left((D(\widetilde B+\J))^2
+D(\widetilde B+\J) C(I-C)^{-1}
D(\widetilde B+\J)\right).
$$
Similar to the proof of Lemma \ref{Fredholm}, to show \reff{Fredeq}  it suffices to prove that
the operators 
$$
\left(D(\widetilde B+\J)\right)^2=D\widetilde BD\widetilde B+D\J D\J+D\widetilde BD\J+D\J D\widetilde B
$$
and 
$
D(\widetilde B+\J)C=D\widetilde B C+D\J C
$
are compact 
from $\CC$ into itself.
Since in the proof of Lemma~\ref{Fred} we already showed  that the operators $D\widetilde BD$ and $D\widetilde BC$ are compact
from $\CC$ into itself,
 it suffices to show that  the operator
$D\J$ is compact
from $\CC$ into itself.
The first component of this operator (and similar for the second component) works as 
$$
[D_1\J w](t,x)=
\int_0^xc(t,x,y)
\int_0^y
\frac{w_1(t-\al(x,y),z)-w_2(t-\al(x,y),z)}{a(z)}\,dzdy
$$
with
$$
c(t,x,y):=-\left[\frac{\partial_2f(y,[Jv](s,y),[Kv](s,y),
[Lv](s,y))}{2\sqrt{a(x)a(y)}}
\exp\left(\frac{1}{2}\int_x^y
\frac{c_+(z,v(s,z))}{a(z)}\,dz\right)\right]_{s=t-\al(x,y)}.
$$
We replace the integration variable $y$ by
$\eta=\widehat{\eta}(t,y,z):=t-\al(x,y)$, hence
$d\eta=-dy/a(y)$.
If $y=\widehat{y}(t,\eta,z)$ is the inverse transformation, then we get
$$
[D_1\J w](t,x)=\int_t^{t-\al(0,x)}\int_0^{\widehat{y}(t,x,\eta)}
c(t,x,\widehat{y}(t,x,\eta))a(\widehat{y}(t,x,\eta))
\frac{w_1(\eta,z)-w_2(\eta,z)}{a(z)}\,dzd\eta.
$$
Hence, the linear operator
$w \in \CC \mapsto D_1\J w \in C([0,1]^2)$ is compact because of the Arzela-Ascoli Theorem.

\end{proof}

\subsection{Proof of Theorem \ref{eqTheorem}}
\label{Proofeq}
 Let $u\in  C^2(\R \times [0,1])$  be a  classical solution to \reff{eq}. 
Then, due to  Lemma \ref{lem:FOS} $(i)$, the function $v \in  C^1(\R \times [0,1];\R^2)$ defined by  \reff{trafo} is a classical solution to system \reff{FOS} 
and, hence, a solution to the abstract equation \reff{abeq}.
If, moreover, $u$ satisfies one of the conditions
\reff{nonreseq}
and
\reff{nonreseq1}, then $v$
satisfies one of the conditions
\reff{nonreseqa}
and
\reff{nonreseqa1}. Because of Lemma \ref{Fredeq1},
Corollary \ref{app1} can be applied to the 
solution $v$ of the abstract equation \reff{abeq}.
Hence,  all partial derivatives 
$\partial_t^kv$, $k \in \N$, exist and are continuous. 
Since $u=Jv$,  all partial derivatives 
$\partial_t^ku$, $k \in \N$, exist and are continuous also. Therefore, assertion $(i)$ of Theorem \ref{eqTheorem} is proved.

In order to prove assertion $ (ii)$ of Theorem \ref{eqTheorem},
suppose that the functions $a$ and $f$ are $C^\infty$-smooth.
Then,  as in Subsection \ref{Proofsys}, we use Lemma \ref{BCD} $(v)$ and show 
that  $v$ is $C^\infty$-smooth.
Again, since $u=Jv$, 
$u$ is $C^\infty$-smooth, as desired.

\section{Appendix }
\label{appendix1}
\setcounter{equation}{0}
\setcounter{theorem}{0}
For given Banach spaces  $U$ and $V$, let  $S_\vp \in {\cal L}(U)$ and 
$T_\vp \in {\cal L}(V)$, with $\vp \in \R$, be
one-parameter $C_0$-groups on $U$ and $V$, respectively, i.e.
$$
\begin{array}{l}
S_\vp \circ S_\psi=S_{\vp + \psi} 
\mbox{ for all } \vp,\psi \in \R,\;
S_0=I,\\
\vp \in \R \mapsto S_\vp u \in U
\mbox{ is continuous for all } u \in U,
\end{array}
$$
and similarly for $T_\vp$. Further, let $\F:U\to V$ be a map such that
\be
\label{equvi}
\F(S_\vp u)=T_\vp\F(u) \mbox{ for all } \vp \in \R \mbox{ and } u \in U.
\ee

The following theorem is due to E. N. Dancer (see \cite[Theorem 1]{Dancer}). Roughly speaking, it claims the following: The map $\gamma \in \R \mapsto S_\gamma u \in U$ is not $C^1$-smooth, in general, but it is if $u$ solves an equivariant equation $\F(u)=0$ with a $C^1$-Fredholm map $\F$.
\begin{theorem}
\label{app}
Let $U$ and $V$ be   Banach spaces.
 Let $\F$ be $C^1$-smooth and $u^0 \in U$ be given  such that 
\be\label{uass}
\F(u^0)=0,
\mbox{ and $\F'(u^0)$  is Fredholm of index zero from $U$ into $V$.}
\ee
If condition \reff{equvi} is fulfilled, then the map $\gamma \in  \R \mapsto S_\gamma u^0 \in U$ is $C^1$-smooth.
\end{theorem}
This theorem can easily be generalized to the $C^\infty$ case as follows:

\begin{corollary}
\label{app1}
Let $U$ and $V$ be   Banach spaces. Let $\F$ be $C^\infty$-smooth and $u^0 \in U$ be given  such that \reff{uass} is satisfied.
If condition \reff{equvi} is fulfilled, then the map $\gamma \in  \R \mapsto S_\gamma u^0 \in U$ is $C^\infty$-smooth.
\end{corollary}

\begin{proof}
We have to show that for any $k \in \N$ the map 
$\vp \in  \R \mapsto S_\vp u^0 \in U$ is $C^k$-smooth. To this end, we  use the induction in $k$.

The assertion for $k=1$ is true due to Theorem \ref{app}.

Doing the induction step, we suppose that, for a fixed $k \in \N$,  the map 
$\vp \in  \R \mapsto S_\vp u^0 \in U$ is $C^k$-smooth and  show that this map is  $C^{k+1}$-smooth.

We denote by $\A:\D(\A) \subseteq U \to U$ the infinitesimal generator of the $C_0$-group $S_\vp$, i.e.
$$
\D(\A):=\{u \in U:\; \vp \in  \R \mapsto S_\vp u \in U \mbox{ is $C^1$-smooth}\},\;
\A u:=\frac{d}{d\vp}S_\vp u |_{\vp=0} \mbox{ for } u \in \D(\A).
$$
Similarly we define  $\D(\A^l)$ and $\A^l$ with $l\ge 2$.
In particular, we have
\begin{eqnarray*}
&&\frac{d}{d\vp}\F(S_\vp u)|_{\vp=0}=\F'(u)\A u
\mbox{ for } u \in \D(\A),\\
&& \frac{d^2}{d\vp^2}\F(S_\vp u)|_{\vp=0}=\F'(u)\A^2u+\F''(u)(\A u,\A u)=0 \quad\mbox{for } u \in \D(\A^2)
\end{eqnarray*}
More precisely, 
there exist  $C^\infty$-maps $\F_l: U^l \to V$, $l \in \N$, such that 
$$
\frac{d^l}{d\vp^l}\F(S_\vp u)|_{\vp=0}=\F'(u)\A^lu+\F_l(u,\A u,\A^2 u, \ldots, \A^{l-1}u)\quad
\mbox{for  }  u \in \D(\A^l). 
$$
On account of \reff{equvi}, for all $\vp \in \R$ and
$u \in \D(\A^{l})$ it holds 
$$
\F_l(S_\vp u,S_\vp \A u,S_\vp\A^2 u, \ldots, S_\vp\A^{l-1} u)=T_\vp\F_l(u,\A u,\A^2 u, \ldots, \A^{l-1}u).
$$ 
hence, $\F(S_\vp u^0)\equiv 0$ yields that
$$
\F'(u^0)\A^lu^0+\F_l(u^0,\A u^0,\A^2 u^0, \ldots, \A^{l-1}u^0)=0\quad
\mbox{for  }  l\le k. 
$$

Now, let us consider the $C^\infty$-map $\G=(\G_0,\G_1,\dots,\G_k):U^{k+1} \to V^{k+1}$ defined by
\begin{eqnarray*}
&&\G_0(u_0,u_1,\ldots,u_k):=\F(u_0),\\
&&\G_j(u_0,u_1,\ldots,u_k):=\F'(u_0)u_j+\F_j(u_0,u_1,\ldots,u_{j-1})\quad
\mbox{for } j\le k.
\end{eqnarray*}
In order to apply Theorem \ref{app} to the equation $\G(u_0,u_1,\ldots,u_k)=0$ in its solution
$$
(u_0,u_1,\ldots,u_k)=(u^0,\A u^0,\ldots,\A^k u^0),
$$
we have to show that the derivative
$\G'(u^0,\A u^0,\ldots,\A^k u^0)$ is Fredholm operator of index zero from $U^{k+1}$ into $V^{k+1}$.
We have
$$
\G'_0(u^0,\A u^0,\ldots,\A^k u^0)
(u_0,u_1,\ldots,u_k)=\F'(u^0)u_0
$$
and, for $j\le k$,
$$
\G'_j(u^0,\A u^0,\ldots,\A^k u^0)
(u_0,u_1,\ldots,u_k):=\F'(u^0)u_j+\sum_{i=0}^{j-1}\partial_i\F_j(u^0,\A u^0,\ldots,\A^{j-1} u^0)u_i.
$$
Hence, $\G'(u^0,\A u^0,\ldots,\A^k u^0)$ is a triangular operator of the type
$$
\G'(u^0,\A u^0,\ldots,\A^k u^0)=
\left[
\begin{array}{cccc}
\F'(u^0) & 0 & 0 & \ldots \\
\partial_0\F_1(u^0) & \F'(u^0) & 0 & \ldots \\
\partial_0\F_2(u^0,\A u^0) & \partial_1\F_2(u^0,\A u^0) & \F'(u^0) & \ldots\\
\ldots &\ldots &\ldots &\ldots \\
\end{array}
\right]
$$ 
By assumption, $\F'(u^0)$ assumption is Fredholm operator of index zero from $U$ into $V$. Hence, there exist linear bounded operators $\J,\KK:U \to V$  
such that $\F'(u^0)=\J+\KK$, that
$\J$ is bijective and that $\KK$ is compact.
Therefore
$$
\G'(u^0,\A u^0,\ldots,\A^k u^0)=
\widetilde{\J}+\widetilde{\KK}
$$ 
with
$$
\widetilde{\J}:=
\left[
\begin{array}{cccc}
\J & 0 & 0 & \ldots \\
\partial_0\F_1(u^0) & \J & 0 & \ldots \\
\partial_0\F_2(u^0,\A u^0) & \partial_1\F_2(u^0,\A u^0) & \J & \ldots\\
\ldots &\ldots &\ldots &\ldots \\
\end{array}
\right],\;
\widetilde{\KK}:=
\left[
\begin{array}{cccc}
\KK & 0 & 0 & \ldots \\
0 & \KK & 0 & \ldots \\
0 & 0 & \KK & \ldots\\
\ldots &\ldots &\ldots &\ldots \\
\end{array}
\right],
$$ 
where $\widetilde{\J}$ is a bijective operator from $U^{k+1}$ to $V^{k+1}$, and $\widetilde{\KK}$ is a compact operator from $U^{k+1}$ into $V^{k+1}$.
Hence, $\G'(u^0,\A u^0,\ldots,\A^k u^0)$ a is Fredholm operator of index zero from $U^{k+1}$ to $V^{k+1}$. Now, Theorem \ref{app} yields that 
$$
\vp \in \R \mapsto (S_\vp u^0,S_\vp\A u^0,\ldots,S_\vp\A^k u^0) \in U^{k+1}
\mbox{ is $C^1$-smooth,}
$$
which means that $\vp \in \R \mapsto S_\vp u^0
\in U$ is $C^{k+1}$-smooth.
\end{proof}

\section*{Acknowledgments}
Irina Kmit was supported by the VolkswagenStiftung Project ``From Modeling and Analysis to Approximation''.

\end{document}